\documentclass[12pt,reqno]{article}

\usepackage[usenames]{color}
\usepackage[colorlinks=true,
linkcolor=webgreen, filecolor=webbrown,
citecolor=webgreen]{hyperref}

\definecolor{webgreen}{rgb}{0,.5,0}
\definecolor{webbrown}{rgb}{.6,0,0}

\usepackage{amsmath}
\usepackage{amssymb}
\usepackage{mathtools}
\usepackage{graphicx}
\usepackage{amscd}
\usepackage{lscape}
\usepackage{tikz}
\usepackage{tikz-cd}
\usepackage{pgfplots}

\usetikzlibrary{matrix}
\usetikzlibrary{fit,shapes}
\usetikzlibrary{positioning, calc}
\tikzset{circle node/.style = {circle,inner sep=1pt,draw, fill=white},
        X node/.style = {fill=white, inner sep=1pt},
        dot node/.style = {circle, draw, inner sep=5pt}
        }
\usepackage{tkz-fct}

\usepackage{amsthm}
\newtheorem{theorem}{Theorem}

\newtheorem{proposition}[theorem]{Proposition}
\newtheorem{corollary}[theorem]{Corollary}

\theoremstyle{definition}
\newtheorem{definition}[theorem]{Definition}
\newtheorem{example}[theorem]{Example}

\usepackage{float}

\usepackage{graphics,amsmath}
\usepackage{amsfonts}
\usepackage{latexsym}
\usepackage{epsf}

\DeclareMathOperator{\Rev}{Rev}

\newcommand{\seqnum}[1]{\href{http://oeis.org/#1}{\underline{#1}}}

\begin{document}

\begin{center}
\vskip 1cm{\LARGE\bf Pascal-like Sprugnoli arrays} \vskip 1cm \large
Paul Barry\\
School of Science\\
South East Technological University\\
Ireland\\
\href{mailto:pbarry@wit.ie}{\tt pbarry@wit.ie}
\end{center}
\vskip .2 in

\begin{abstract} In this note, we look at the structure and properties of palindromic or Pascal-like Sprugnoli arrays. We show that there are two closely related families of these arrays. We give closed form expressions for the elements of these families, and in each case, we describe the form of the inverse arrays. Finally, we consider the arrays modulo $2$ and the resulting arithmetic sequences.
\end{abstract}

\section{Introduction} The Sprugnoli group \cite{Spru} is a generalization of the Riordan group \cite{book1, book2, SGWW}. Elements of the Sprugnoli group are defined by a triple of power series $(g, f_1, f_2)$, while elements of the Riordan group are defined by a pair of power series $(g, f)$. We briefly review the definition of the Riordan group so as to place the Sprugnoli group in context.

We recall that $\mathcal{F}_0 = \{ \sum_{n=0}^{\infty}a_n x^n\,|\, a_0 \ne 0\}$ and that $\mathcal{F}_1 = \{ \sum_{n=0}^{\infty}a_n x^n\,|\, a_0 =0, a_1 \ne 0\}$ and in general
$$\mathcal{F}_r = \{\sum_{n=r}^{\infty}a_n x^n\,|\, a_r \ne 0\}.$$ Elements of $\mathcal{F}_0$ are multiplicatively invertible, and elements of $\mathcal{F}_1$ are compositionally invertible, given suitable ground rings $R$ for $a_n \in R$.

The Riordan group is then the group of pairs $(g, f) \in \mathcal{F}_0 \times \mathcal{F}_1$ with the following product rule
\begin{equation}\label{product}(g(x), f(x))\cdot (u(x), v(x))=  (g(x)u(f(x)), v(f(x))\end{equation} and inverse
$$(g, f)^{-1}= \left(\frac{1}{g(\bar{f})}, \bar{f}(x)\right),$$ where $\bar{f}$ is the compositional inverse of $f \in \mathcal{F}_0$ (that is, $\bar{f}$ is the solution $u$ of the equation $f(u)=x$ for which $u(0)=0$). The identity of this group is $(1,x)$. To each element of this group we can associate in a unique way a lower-triangular matrix $(t_{n,k})$ with entries in the ground ring $R$ by means of
$$t_{n,k}=[x^n] g(x)f(x)^k,$$ where $[x^n]$ is the functional on $R[[x]]$ that extracts the coefficient of $x^n$. Under this correspondence, the product (\ref{product}) corresponds to ordinary matrix multiplication. The columns of $(g,f)$ have their generating functions given by the geometric series of power series
$$g, gf, gf^2, gf^3, gf^4, \ldots.$$ The bi-variate generating function of the Riordan array $(g(x), f(x))$ is given by $\frac{g(x)}{1-yf(x)}$.
\begin{example} The Riordan array $(g(x),f(x))=\left(\frac{1}{1-x}, \frac{x}{(1-x)^2}\right)$ begins
$$\left(\begin{array}{ccccccc}
1 & 0 & 0 & 0 & 0 & 0 & 0 \\
1 & 1 & 0 & 0 & 0 & 0 & 0 \\
1 & 3 & 1 & 0 & 0 & 0 & 0 \\
1 & 6 & 5 & 1 & 0 & 0 & 0 \\
1 & 10 & 15 & 7 & 1 & 0 & 0 \\
1 & 15 & 35 & 28 & 9 & 1 & 0 \\
1 & 21 & 70 & 84 & 45 & 11 & 1\\
\end{array}\right).$$
We have $t_{n,k}=\binom{n+k}{2k}$.
The generating function of this array is given by
$$\frac{\frac{1}{1-x}}{1-y\frac{x}{(1-x)^2}}=\frac{1-x}{1-x(y+2)+x^2}.$$
When $y=1$, we obtain the generating function $\frac{1-x}{1-3x+x^2}$ of the row sums of this matrix. Thus the row sums begin
$$1, 2, 5, 13, 34, 89, 233, 610, 1597, 4181, 10946,\ldots,$$ or $F_{2n+1}$ \seqnum{A122367}.
\end{example}  The product law follows from the following result, called the ``fundamental theorem of Riordan arrays'', which details how a Riordan array operates on a power series. We have
$$(g(x), f(x))\cdot h(x)= g(x)h(f(x)).$$ This is sometimes called a weighted composition.

A similar law holds for so-called ``vertically stretched'' Riordan arrays \cite{LI}. These are matrices corresponding to pairs of power series of the form $(g, xf(x))$ where as usual, $f(x) \in \mathcal{F}_1$, and so $xf(x) \in \mathcal{F}_2$. Again, we have a weighted composition action on power series given by
$$ (g(x), xf(x))\cdot h(x)= g(x)h(xf(x)).$$
The generating function of the stretched Riordan array $(g, xf(x))$  is given by
$$(g(x),xf(x))\cdot \frac{1}{1-xy}=\frac{g(x)}{1-yxf(x)}.$$ In particular, the generating function of the row sums of the stretched Riordan array $(g,xf)$ will have generating function $\frac{g(x)}{1-xf(x)}$.
\begin{example} The stretched Riordan array $\left(\frac{1}{1-x}, \frac{x^2}{1-x-x^2}\right)$ begins
$$\left(\begin{array}{ccccccccc}
1 & 0 & 0 & 0 & 0 & 0 & 0 & 0 & 0 \\
1 & 0 & 0 & 0 & 0 & 0 & 0 & 0 & 0 \\
1 & 1 & 0 & 0 & 0 & 0 & 0 & 0 & 0 \\
1 & 2 & 0 & 0 & 0 & 0 & 0 & 0 & 0 \\
1 & 4 & 1 & 0 & 0 & 0 & 0 & 0 & 0 \\
1 & 7 & 3 & 0 & 0 & 0 & 0 & 0 & 0 \\
1 & 12 & 8 & 1 & 0 & 0 & 0 & 0 & 0 \\
1 & 20 & 18 & 4 & 0 & 0 & 0 & 0 & 0 \\
1 & 33 & 38 & 13 & 1 & 0 & 0 & 0 & 0
\end{array}\right).$$ We have, for instance,
$$\left(\begin{array}{ccccccccc}
1 & 0 & 0 & 0 & 0 & 0 & 0 & 0 & 0 \\
1 & 0 & 0 & 0 & 0 & 0 & 0 & 0 & 0 \\
1 & 1 & 0 & 0 & 0 & 0 & 0 & 0 & 0 \\
1 & 2 & 0 & 0 & 0 & 0 & 0 & 0 & 0 \\
1 & 4 & 1 & 0 & 0 & 0 & 0 & 0 & 0 \\
1 & 7 & 3 & 0 & 0 & 0 & 0 & 0 & 0 \\
1 & 12 & 8 & 1 & 0 & 0 & 0 & 0 & 0 \\
1 & 20 & 18 & 4 & 0 & 0 & 0 & 0 & 0 \\
1 & 33 & 38 & 13 & 1 & 0 & 0 & 0 & 0
\end{array}\right)\left( \begin{array}{c}1\\1\\2\\3\\5\\8\\13\\21\\34\end{array}\right)=\left( \begin{array}{c}1\\1\\2\\3\\7\\14\\32\\69\\159\end{array}\right),$$ where the last sequence has generating function
$$\left(\frac{1}{1-x},\frac{x^2}{1-x-x^2}\right)\cdot \frac{1}{1-x-x^2}=\frac{\frac{1}{1-x}}{1-\frac{x^2}{1-x-x^2}-(\frac{x^2}{1-x-x^2})^2},$$ which is equal to
$$\frac{(1-x-x^2)^2}{(1-x)(1-2x+3x^3+x^4)}.$$
\end{example}

\begin{definition} The \emph{Sprugnoli group} is the set with elements $(g, f_1, f_2)$ with $g(x) \in \mathcal{F}_0$, $f_1(x) \in \mathcal{F}_1$, and $f_2(x) \in \mathcal{F}_1$ and
$f_2 \in xR[[x^2]]$ (thus $f_2$ is an odd power series). The element $(g, f_1, f_2)$ of this group has the matrix representation
$$t_{n,k} = [x^n] g(x)f_1(x)^{k \bmod 2} (xf_2(x))^{\lfloor \frac{k}{2} \rfloor}.$$
\end{definition}
Thus we have
$$
t_{n,k} =
\begin{cases}
  [x^n]g(x)(xf_2(x))^m, & k=2m, \\
  [x^n]g(x)f_1(x)(xf_2(x))^m,  & k=2m+1.
\end{cases}
$$
The columns of this matrix then have generating functions given by the following products of generating functions
\begin{align*}g(x), g(x)f_1(x),&\, g(x)(xf_2(x)), g(x)f_1(x)(xf_2(x)), g(x)(xf_2(x))^2,\\
& g(x)f_1(x)(xf_2(x))^2, g(x)(xf_2(x))^3, g(x)f_1(x)(xf_2(x))^3,\ldots\end{align*} We can represent this sequence by the following schema.
$$
\begin{array}{cccccccc}
g &g & g & g & g & g & g & g\\
1 &f_1 & 1 & f_1 & 1 & f_1 & 1 & f_1 \\
1 &1 & f_2 & f_2 & f_2^2 & f_2^2 & f_2^3 & f_2^3 \\
1 & 1 & x & x & x^2 & x^2 & x^3 & x^3 \\
- & - & - & - & - & - & - & -\\
0 & 1 & 2 & 3 & 4 & 5 & 6 & 7 \\
\end{array}
$$ Here, the final row gives the order of the product of the terms above it, showing that this associated matrix is lower-triangular. We can represent this array as the sum of two matrices as follows. The first matrix is the matrix whose columns are generated by the power series
$$g(x), 0, g(x)(xf_2(x)), 0, g(x)(xf_2(x))^2, \ldots,$$ which is a horizontal aeration of the stretched Riordan array $(g(x), xf_2(x))$, and the second matrix is the matrix whose columns are generated by the power series
$$0,g(x)f_1(x),0, g(x)f_1(x)(xf_2(x)), 0, g(x)f_1(x)(xf_2(x))^2,0,\ldots,$$ which is the horizontal aeration of the stretched Riordan array $(g(x)f_1(x), xf_2(x))$ (but note that this matrix starts with a zero row, due to the fact that $f_1 \in \mathcal{F}_1$).\

\begin{example} We consider the Sprugnoli array $\left(\frac{1}{1-x}, \frac{x(1+x)}{1-x}, \frac{x}{1-x^2}\right)$. This array begins
$$\left(\begin{array}{ccccccccc}
1 & 0 & 0 & 0 & 0 & 0 & 0 & 0 & 0 \\
1 & 1 & 0 & 0 & 0 & 0 & 0 & 0 & 0 \\
1 & 3 & 1 & 0 & 0 & 0 & 0 & 0 & 0 \\
1 & 5 & 1 & 1 & 0 & 0 & 0 & 0 & 0 \\
1 & 7 & 2 & 3 & 1 & 0 & 0 & 0 & 0 \\
1 & 9 & 2 & 6 & 1 & 1 & 0 & 0 & 0 \\
1 & 11 & 3 & 10 & 3 & 3 & 1 & 0 & 0 \\
1 & 13 & 3 & 15 & 3 & 7 & 1 & 1 & 0 \\
1 & 15 & 4 & 21 & 6 & 13 & 4 & 3 & 1
\end{array}\right).$$ This is the sum of the matrices
$$\left(\begin{array}{ccccccccc}
1 & 0 & 0 & 0 & 0 & 0 & 0 & 0 & 0 \\
1 & 0 & 0 & 0 & 0 & 0 & 0 & 0 & 0 \\
1 & 0 & 1 & 0 & 0 & 0 & 0 & 0 & 0 \\
1 & 0 & 1 & 0 & 0 & 0 & 0 & 0 & 0 \\
1 & 0 & 2 & 0 & 1 & 0 & 0 & 0 & 0 \\
1 & 0 & 2 & 0 & 1 & 0 & 0 & 0 & 0 \\
1 & 0 & 3 & 0 & 3 & 0 & 1 & 0 & 0 \\
1 & 0 & 3 & 0 & 3 & 0 & 1 & 0 & 0 \\
1 & 0 & 4 & 0 & 6 & 0 & 4 & 0 & 1
\end{array}\right)+
\left(\begin{array}{ccccccccc}
0 & 0 & 0 & 0 & 0 & 0 & 0 & 0 & 0 \\
0 & 1 & 0 & 0 & 0 & 0 & 0 & 0 & 0 \\
0 & 3 & 0 & 0 & 0 & 0 & 0 & 0 & 0 \\
0 & 5 & 0 & 1 & 0 & 0 & 0 & 0 & 0 \\
0 & 7 & 0 & 3 & 0 & 0 & 0 & 0 & 0 \\
0 & 9 & 0 & 6 & 0 & 1 & 0 & 0 & 0 \\
0 & 11 & 0 & 10 & 0 & 3 & 0 & 0 & 0 \\
0 & 13 & 0 & 15 & 0 & 7 & 0 & 1 & 0 \\
0 & 15 & 0 & 21 & 0 & 13 & 0 & 3 & 0
\end{array}\right).$$
The first is a horizontal aeration of the stretched Riordan array $\left(\frac{1}{1-x},\frac{x}{1-x^2}\right)$, while the second one
is a horizontal aeration of the stretched Riordan array $\left(\frac{1}{1-x} \frac{x(1+x)}{1-x},\frac{x}{1-x^2}\right)=\left(\frac{x(1+x)}{(1-x)^2},\frac{x}{1-x^2}\right)$ with an extra initial column of zeros.
\end{example}
A special element of this set is given by $(1,x,x)$. Its columns are generated by the sequence $1,x,x^2,x^3,\ldots$ and so its matrix representation is  given by the usual identity matrix.

Elements of this set will be referred to as Sprugnoli arrays, in reference to their matrix representation.
\begin{proposition}
The bi-variate generating function of the Sprugnoli array $(g, f_1, f_2)$ is given by
$$\frac{g(x)}{1-y^2xf_2(x)}+\frac{yg(x)f_1(x)}{1-y^2xf_2(x)}=\frac{g(x)(1+yf_1(x))}{1-y^2xf_2(x)}.$$
\end{proposition}
\begin{proof} The array is given by the sum of the horizontal aeration of the stretched Riordan array $(g(x), xf_2(x))$ and the horizontal aeration of the stretched Riordan array $(g(x)f_1(x), xf_2(x))$. Interpreting this in terms of generating functions gives the result.
\end{proof}
\begin{corollary} The row sums of the Sprugnoli array $(g, f_1, f_2)$ have generating function $\frac{g(x)(1+f_1(x))}{1-xf_2(x)}$.
\end{corollary}
\begin{proof} This results by setting $y=1$ in the generating function of the array.
\end{proof}
\begin{corollary} The diagonal sums of the Sprugnoli array $(g, f_1, f_2)$ have generating function $\frac{g(x)(1+xf_1(x))}{1-x^2f_2(x)}$.
\end{corollary}
\begin{proof} This results by setting $y=x$ in the generating function of the array.
\end{proof}

\section{The ``fundamental theorem'' of Sprugnoli arrays}
The following example illustrates the result that follows.

\begin{proposition} Let $(g, f_1, f_2)$ be a Sprugnoli array, and let $h(x)$ be a power series $h_n(x)=\sum_{n=0}^{\infty} a_n x^n$. We let
$$h^e(x)=\sum_{n=0}^{\infty} a_{2n}x^n \quad\text{and}\quad h^o(x)=\sum_{n=0}^{\infty} a_{2n+1}x^n$$ be the even and odd bisections of $h(x)$. Then we have
$$(g(x), f_1(x), f_2(x))\cdot h(x)=g(x)h^e(xf_2(x))+g(x)f_1(x)h^o(xf_2(x)).$$
\end{proposition}
\begin{proof} We partition the matrix representing $(g, f_1, f_2)$ as
$$(g,0,g(xf_2),0,g(xf_2)^2,0,g(xf_2)^3,0,\ldots) + (0, gf_1, 0, gf_1(xf_2), 0, gf_1(xf_2)0,\ldots).$$
These matrices operate, respectively, on $(a_0,0,a_2,0,a_4,0,\ldots)$ and $(0,a_1,0,a_3,0,a_5,0,\ldots)$.
In terms of Riordan arrays, this is equivalent, after compression, to the stretched Riordan array $(g, xf_2)$ operating on $h^e$, and the
stretched Riordan array $(gf_1, xf_2)$ operating on $h^o(x)$. The result follows from this.
\end{proof}
We recall that $g^e(x)=\frac{g(\sqrt{x})+g(-\sqrt{x})}{2}$ and $g^o(x)=\frac{g(\sqrt{x})-g(-\sqrt{x})}{2\sqrt{x}}$.
\begin{example} We have
\begin{align*}(1, x, x) \cdot h(x)&= 1.h^e(x^2)+1.x.h^o(x^2)\\
&=h^e(x^2)+xh^o(x^2)\\
&=h(x).\end{align*}
\end{example}
\section{The product rule}
We are now in a position to say what we mean by the product of two Sprugnoli arrays. We have the following.
\begin{definition} We define the product of two Sprugnoli arrays $(g, f_1, f_2)$ and $(u, v_1, v_2)$ as follows.
\begin{equation}\label{prod}(g, f_1, f_2) \cdot (u, v_1, v_2)= \left((g,f_1,f_2)\cdot u, \frac{(g,f_1,f_2)\cdot u v_1}{(g,f_1,f_2)\cdot u}, \frac{1}{x}.\frac{(g,f_1,f_2)\cdot uxv_2}{ (g,f_1,f_2)\cdot u}\right).\end{equation}
\end{definition}
We have
\begin{equation}\label{prod2}(g, f_1, f_2) \cdot (u, v_1, v_2)= \left((g,f_1,f_2)\cdot u, \frac{(g,f_1,f_2)\cdot u v_1}{(g,f_1,f_2)\cdot u}, \frac{1}{x} \sqrt{xf_2}v_2(\sqrt{xf_2})\right).\end{equation}
Furthermore, the product rule (\ref{prod}), when we look to the matrix representation of Sprugnoli arrays, coincides with matrix multiplication.

\section{The inverse of a Sprugnoli array}

We have
$$(g,f_1,f_2)^{-1}=\left(1,\frac{x-f_1^e(\left(\overline{\sqrt{xf_2}}\right)^2)}{f_1^o(\left(\overline{\sqrt{xf_2}}\right)^2)},\frac{1}{x}\left(\overline{\sqrt{xf_2}}\right)^2\right)\cdot \left(\frac{1}{g},x,x\right).$$

\section{Pascal-like arrays}
Pascal's triangle, with general element $\left(\binom{n}{k}\right)$ is the number triangle that begins
$$\left(\begin{array}{ccccccc}
1 & 0 & 0 & 0 & 0 & 0 & 0 \\
1 & 1 & 0 & 0 & 0 & 0 & 0 \\
1 & 2 & 1 & 0 & 0 & 0 & 0 \\
1 & 3 & 3 & 1 & 0 & 0 & 0 \\
1 & 4 & 6 & 4 & 1 & 0 & 0 \\
1 & 5 & 10 & 10 & 5 & 1 & 0 \\
1 & 6 & 15 & 20 & 15 & 6 & 1
\end{array}\right).$$ This matrix is palindromic, or equivalently, centrally symmetric, with ones in the first column and on the diagonal. We shall describe number triangles with this property as being \textbf{\emph{Pascal-like}}. Pascal's triangle is given by the Riordan array $\left(\frac{1}{1-x}, \frac{x}{1-x}\right)$. Riordan arrays that are Pascal-like are of the form \cite{Intcon}
$$\left(\frac{1}{1-x}, \frac{x(1+rx)}{1-x}\right).$$ We now turn our attention to Sprugnoli arrays that are Pascal-like. A first example is the Sprugnoli array $\left(\frac{1}{1-x}, \frac{x}{1+x}, \frac{x}{1-x^2}\right)$, which begins
$$\left(\begin{array}{cccccccccc}
1 & 0 & 0 & 0 & 0 & 0 & 0 & 0 & 0 & 0 \\
1 & 1 & 0 & 0 & 0 & 0 & 0 & 0 & 0 & 0 \\
1 & 0 & 1 & 0 & 0 & 0 & 0 & 0 & 0 & 0 \\
1 & 1 & 1 & 1 & 0 & 0 & 0 & 0 & 0 & 0 \\
1 & 0 & 2 & 0 & 1 & 0 & 0 & 0 & 0 & 0 \\
1 & 1 & 2 & 2 & 1 & 1 & 0 & 0 & 0 & 0 \\
1 & 0 & 3 & 0 & 3 & 0 & 1 & 0 & 0 & 0 \\
1 & 1 & 3 & 3 & 3 & 3 & 1 & 1 & 0 & 0 \\
1 & 0 & 4 & 0 & 6 & 0 & 4 & 0 & 1 & 0 \\
1 & 1 & 4 & 4 & 6 & 6 & 4 & 4 & 1 & 1
\end{array}\right).$$ This is sequence \seqnum{A051159} in the On-Line Encyclopedia of Integer Sequences (OEIS) \cite{SL1, SL2}. Its general term is given by
$$t_{n,k}=\binom{\lfloor \frac{n}{2} \rfloor}{\lfloor \frac{k}{2} \rfloor} \binom{n \bmod 2}{k \bmod 2}.$$
The generating function of this triangle is given by
$$\frac{g(x)(1+yf_1(x))}{1-y^2xf_2(x)}=\frac{1+(y+1)x}{1-(y^2+1)x^2}.$$
Setting $y=1$ gives us the generating function of the row sums. We obtain $\frac{1+2x}{1-2x^2}$, the generating function of the sequence  \seqnum{A060546} that begins
$$1, 2, 2, 4, 4, 8, 8, 16, 16, 32, 32,\ldots.$$ Setting $y=x$ gives us the generating function of the diagonal sums of this triangle. We obtain $\frac{1+x+x^2}{1-x^2-x^4}$, the generating function of the sequence (essentially \seqnum{A053602}) that begins
$$1, 1, 2, 1, 3, 2, 5, 3, 8, 5, 13,\ldots.$$  This sequence is an interleaving of $F_{n+2}$ and $F_{n+1}$, where $F_n$ is the $n$-th Fibonacci number \seqnum{A000045}.
\begin{example}
When we look at the inversion of the above matrix we get an interesting result. By the inversion of this matrix, we refer to the lower triangular matrix whose generating function is given by
$$\frac{1}{x} \Rev_x{\frac{1+x+xy}{1-x^2-y^2x^2}}=\frac{\sqrt{1+4(1+y)x+4(1+y^2)x^2}-1}{2(1+y+(1+y^2)x)}.$$
This expands to give us the sequence of polynomials
$$1, -y - 1, y^2 + 4y + 1, - y^3 - 11y^2 - 11y - 1, y^4 + 26y^3 + 58y^2 + 26y + 1,\ldots,$$ with coefficient array that begins
$$\left(\begin{array}{cccccccccc}
1 & 0 & 0 & 0 & 0 & 0 & 0 & 0 & 0 \\
 -1 & -1 & 0 & 0 & 0 & 0 & 0 & 0 & 0 \\
 1 & 4 & 1 & 0 & 0 & 0 & 0 & 0 & 0 \\
 -1 & -11 & -11 & -1 & 0 & 0 & 0 & 0 & 0 \\
 1 & 26 & 58 & 26 & 1 & 0 & 0 & 0 & 0 \\
 -1 & -57 & -226 & -226 & -57 & -1 & 0 & 0 & 0 \\
 1 & 120 & 747 & 1296 & 747 & 120 & 1 & 0 & 0 \\
 -1 & -247 & -2229 & -5907 & -5907 & -2229 & -247 & -1 & 0 \\
 1 & 502 & 6208 & 23242 & 35294 & 23242 & 6208 & 502 & 1
 \end{array}\right).$$ We regard this matrix as the inversion of the Sprugnoli array in question. We then have the result that
 $$\left(\frac{1-(y+1)x}{1-2(1+y)x+2yx^2},\frac{x}{1-2(1+y)x+2yx^2}\right)^{-1}=$$
$$\left(\frac{1}{x} \Rev_x{\frac{1+x+xy}{1-x^2-y^2x^2}}, \frac{1+2(1+y)x-\sqrt{1+4(1+y)x+4(1+y^2)x^2}}{4xy}\right).$$ This shows that the inversion of the above Pascal-like Sprugnoli matrix gives the moment sequence for the generalized orthogonal polynomials whose coefficient array is given by the Riordan array $\left(\frac{1-(y+1)x}{1-2(1+y)x+2yx^2},\frac{x}{1-2(1+y)x+2yx^2}\right)$.
 \end{example}
\begin{example} The Sprugnoli array $\left(\frac{1}{1-x}, x, \frac{x}{1-x^2}\right)$ begins
$$\left(\begin{array}{cccccccccc}
1 & 0 & 0 & 0 & 0 & 0 & 0 & 0 & 0 \\
1 & 1 & 0 & 0 & 0 & 0 & 0 & 0 & 0 \\
1 & 1 & 1 & 0 & 0 & 0 & 0 & 0 & 0 \\
1 & 1 & 1 & 1 & 0 & 0 & 0 & 0 & 0 \\
1 & 1 & 2 & 1 & 1 & 0 & 0 & 0 & 0 \\
1 & 1 & 2 & 2 & 1 & 1 & 0 & 0 & 0 \\
1 & 1 & 3 & 2 & 3 & 1 & 1 & 0 & 0 \\
1 & 1 & 3 & 3 & 3 & 3 & 1 & 1 & 0 \\
1 & 1 & 4 & 3 & 6 & 3 & 4 & 1 & 1
\end{array}\right).$$
Operating with this on the positive Fibonacci numbers $F_{n+1}$ yields the sequence with generating function $\frac{1+2x-x^2-3x^3-x^4}{1-5x^2+5x^4}$ that begins
$$1, 2, 4, 7, 14, 25, 50, 90, 180, 325, 650,\ldots.$$ The bisection that begins $1,4,14,50,\ldots$ has generating function $\frac{1-x-x^2}{1-5x+5x^2}$. This is essentially \seqnum{A153367}. The odd bisection $2,7,25,90,\ldots$ has generating function $\frac{2-3x}{1-5x+5x^2}$. This is the binomial transform of $F_{2n+3}$. A closely related sequence is \seqnum{A052936}.
\end{example}
We have the following proposition concerning Pascal-like Sprugnoli arrays.
\begin{proposition} The Sprugnoli arrays $\left(\frac{1}{1-x}, x, \frac{x(1+rx^2)}{1-x^2}\right)$ and $\left(\frac{1}{1-x}, \frac{x(1+sx)}{1+x}, \frac{x(1+rx^2)}{1-x^2}\right)$ are Pascal-like.
\end{proposition}
\begin{proof} The proof consists of finding a closed form expression for the elements $t_{n,k}$ of the array, and then to show that $t_{n,k}=t_{n,n-k}$. We take the case of the array $\left(\frac{1}{1-x}, \frac{x(1+ax)}{1+x}, \frac{x(1+rx^2)}{1-x^2}\right)$. We have
$$t_{n,k}=[x^n]\frac{1}{1-x}\left(\frac{x(1+sx)}{1+x}\right)^{k \bmod 2}\left(\frac{x^2(1+rx^2)}{1-x^2}\right)^{\lfloor \frac{k}{2} \rfloor}.$$
When $k$ is even, $k=2m$, say, this gives
\begin{align*}
t_{n,k}&=[x^n]\frac{1}{1-x} x^{2m}\frac{(1+rx^2)^m}{(1-x^2)^m}\\
&=[x^{n-k}] \sum_{i=0}^{\infty}x^i \sum_{j=0}^m \binom{m}{j}r^jx^{2j}\sum_{l=0} \binom{m+l-1}{l}x^{2l}.\end{align*}
When $k$ is odd, $k=2m+1$, say, we have
\begin{align*}
t_{n,k}&=[x^n]\frac{1}{1-x}\left(\frac{x(1+sx)}{1-x}\right)^{(2m+1 \bmod 2)}\left(\frac{x^2(1+rx^2)}{1-x^2}\right)^{\lfloor \frac{2m+1}{2} \rfloor}\\
&=[x^n]\frac{x(1+sx)}{1-x^2}\left(\frac{x^2(1+rx^2)}{1-x^2}\right)^m\\
&=[x^{n-2m-1}](1+sx) \frac{(1+rx^2)^m}{(1-x^2)^{m+1}}\\
&=[x^{n-k}] \sum_{j=0}^m \binom{m}{j}r^jx^{2j} \sum_{l=0}\binom{m+l}{l}x^{2l}+s[x^{n-k-1}] \sum_{j=0}^m \binom{m}{j}r^jx^{2j} \sum_{l=0}\binom{m+l}{l}x^{2l}.\end{align*}
We find that, for $k$ even, with $0\le k\le n$, we have
\begin{align*}t_{n,k}&=\sum_{j=0}^{k/2} \binom{\frac{k}{2}}{j}r^j\left(\binom{\frac{n-2j}{2}}{\frac{k}{2}}\frac{1+(-1)^{n-k}}{2}+\binom{\frac{n-2j-1}{2}}{\frac{k}{2}}\frac{1-(-1)^{n-k}}{2}\right)\\
&=\sum_{j=0}^{k/2}\binom{\frac{k}{2}}{j}r^j \sum_{i=0}^{n-k-2j} \binom{\frac{n-2j-i}{2}-1}{\frac{n-2j-k-i}{2}}\frac{1+(-1)^{n-k-i}}{2}.\end{align*}
and for $k$ odd, with $0\le k\le n$, we have
$$t_{n,k}=\sum_{j=0}^{(k-1)/2}\binom{\frac{k-1}{2}}{j}\{\binom{\frac{n-2j-1}{2}}{\frac{n-k-2j}{2}}\frac{1+(-1)^{n-k}}{2}+s\binom{\frac{n-2j-2}{2}}{\frac{n-k-2j-1}{2}}\frac{1+(-1)^{n-k-1}}{2}\}.$$
As $0\le k \le n$, changing $k$ to $n-k$ leads to the same values for the summations.
\end{proof}

\section{Generating functions, row sums and diagonal sums}
The generating function of the Pascal-like Sprugnoli array $\left(\frac{1}{1-x}, x, \frac{x(1+rx^2)}{1-x^2}\right)$ is given by
$$G(x,y)=\frac{(1+x)(1+xy)}{1-(1+y^2)x^2-rx^4y^2}.$$
The generating function of the row sums of this matrix are then given by
$$G(x,1)=\frac{(1+x)^2}{1-2x^2-rx^4},$$ while the diagonal sums have their generating function given by
$$G(x,x)=\frac{(1+x)(1+x^2)}{1-x^2-x^4-rx^6}.$$
The generating function of the Pascal-like Sprugnoli array $\left(\frac{1}{1-x}, \frac{x(1+sx)}{1+x}, \frac{x(1+rx^2)}{1-x^2}\right)$ is given by
$$G(x,y)=\frac{1+(1+y)x+syx^2}{1-(1+y^2)x^2-ry^2x^4}.$$
The generating function of the row sums of this matrix is then given by
$$G(x,1)=\frac{1+2x+sx^2}{1-2x^2-rx^4},$$ while the diagonal sums have their generating function given by
$$G(x,x)=\frac{1+x+x^2+sx^3}{1-x^2-x^4-rx^6}.$$
This matrix begins 
\[
\begin{pmatrix}
1 & 0 & 0 & 0 & 0 & 0 & 0 \\
1 & 1 & 0 & 0 & 0 & 0 & 0 \\
1 & s & 1 & 0 & 0 & 0 & 0 \\
1 & 1 & 1 & 1 & 0 & 0 & 0 \\
1 & s & r+2 & s & 1 & 0 & 0 \\
1 & 1 & r+2 & r+2 & 1 & 1 & 0 \\
1 & s & 2r+3 & s(r+2) & 2r+3 & s & 1
\end{pmatrix}
\].
\section{The Sprugnoli-Delannoy matrices}
The Delannoy matrix \seqnum{A008288}, viewed by diagonals as a triangle, is given by the Riordan array $\left(\frac{1}{1-x}, \frac{x(1+x)}{1-x}\right)$. The corresponding Sprugnoli matrices are then $$\left(\frac{1}{1-x}, x, \frac{x(1+x^2)}{1-x^2}\right)\quad\text{and}\quad \left(\frac{1}{1-x}, \frac{x}{1+x}, \frac{x(1+x^2)}{1-x^2}\right).$$
These matrices begin, respectively,
$$\left(\begin{array}{cccccccccc}
1 & 0 & 0 & 0 & 0 & 0 & 0 & 0 & 0 \\
1 & 1 & 0 & 0 & 0 & 0 & 0 & 0 & 0 \\
1 & 1 & 1 & 0 & 0 & 0 & 0 & 0 & 0 \\
1 & 1 & 1 & 1 & 0 & 0 & 0 & 0 & 0 \\
1 & 1 & 3 & 1 & 1 & 0 & 0 & 0 & 0 \\
1 & 1 & 3 & 3 & 1 & 1 & 0 & 0 & 0 \\
1 & 1 & 5 & 3 & 5 & 1 & 1 & 0 & 0 \\
1 & 1 & 5 & 5 & 5 & 5 & 1 & 1 & 0 \\
1 & 1 & 7 & 5 & 13 & 5 & 7 & 1 & 1
\end{array}\right)\quad\text{and}\quad\left(\begin{array}{cccccccccc} 1 & 0 & 0 & 0 & 0 & 0 & 0 & 0 & 0 \\
1 & 1 & 0 & 0 & 0 & 0 & 0 & 0 & 0 \\
1 & 0 & 1 & 0 & 0 & 0 & 0 & 0 & 0 \\
1 & 1 & 1 & 1 & 0 & 0 & 0 & 0 & 0 \\
1 & 0 & 3 & 0 & 1 & 0 & 0 & 0 & 0 \\
1 & 1 & 3 & 3 & 1 & 1 & 0 & 0 & 0 \\
1 & 0 & 5 & 0 & 5 & 0 & 1 & 0 & 0 \\
1 & 1 & 5 & 5 & 5 & 5 & 1 & 1 & 0 \\
1 & 0 & 7 & 0 & 13 & 0 & 7 & 0 & 1\end{array}\right).
$$
These have the respective generating functions
$$\frac{1+x+xy+x^2y}{1-(y^2+1)x^2-y^2x^4} \quad\text{and}\quad \frac{1+x+xy}{1-(y^2+1)x^2-y^2x^4}.$$
The row sums of the Sprugnoli array $\left(\frac{1}{1-x}, x, \frac{x(1+x^2)}{1-x^2}\right)$ have generating function
$\frac{(1+x)^2}{1-2^2-x^4}$. This sequence begins
$$1, 2, 3, 4, 7, 10, 17, 24, 41, 58, 99, 140, 239,\ldots,$$ and its bisections have generating functions
$\frac{1+x}{1-2x-x^2}$ or \seqnum{A078057} (essentially the Pell-Lucas numbers), and $\frac{2}{1-2x-x^2}$, which expands to give twice the Pell numbers \seqnum{A000129}. The diagonal sums of this matrix have generating function given by $\frac{(1+x)(1+x^2)}{1-x^2-x^4-x^6}$. This expands to give the sequence that begins
$$1, 1, 2, 2, 3, 3, 6, 6, 11, 11, 20, 20, 37,\ldots.$$ This is a doubled version of the tribonacci related sequence \seqnum{A047081}
$$1, 2, 3, 6, 11, 20, 37, 68, 125, 230, 423,\ldots$$ with generating function $\frac{1+x}{1-x-x^2-x^3}$.

The row sums of the Sprugnoli-Delannoy array $\left(\frac{1}{1-x}, \frac{x}{1+x}, \frac{x(1+x^2)}{1-x^2}\right)$ have their generating function given by $\frac{1+2x}{1-2x^2-x^4}$, and they begin
$$1, 2, 2, 4, 5, 10, 12, 24, 29, 58, 70,\ldots.$$ The bisections in this case are the Pell numbers and twice the Pell numbers, respectively. The diagonal sums of this array have generating function $\frac{1+x+x^2}{1-x^2-x^4-x^6}$. This expands to give the sequence that begins
$$1, 1, 2, 1, 3, 2, 6, 4, 11, 7, 20, 13, 37,\ldots.$$ The even bisection of this sequence is \seqnum{A047081}, while the odd bisection, with generating function $\frac{1}{1-x-x^2-x^3}$, is given by the tribonacci numbers \seqnum{A000073}.

\section{Inverses}
We shall denote the inverse of a Sprugnoli array $(g, f_1,f_2)$ by $(w, s_1, s_2)$. With regard to Pascal-like Sprugnoli matrices,
for the case $r=0$, we have \cite{Spru}
\begin{align*}\left(\frac{1}{1-x}, x, \frac{x}{1-x^2}\right)^{-1}&=\left(1-x, \frac{x(1-x+x^2)}{(1-x)(1+x^2)}, \frac{x}{1+x^2}\right)\\
\left(\frac{1}{1-x}, \frac{x}{1+x}, \frac{x}{1-x^2}\right)^{-1}&=\left(\frac{1-x}{1+x^2}, \frac{x}{1-x}, \frac{x}{1+x^2}\right).\end{align*}
The Pascal-like Sprugnoli matrices all have $f_2(x)=\frac{x(1+rx^2)}{1-x^2}$ so $s_2(x)$ will be the same for all. For this, we must solve the equation
$$u\sqrt{\frac{1+ru^2}{1-u^2}}=x,$$ or
$$u^2 \frac{1+ru^2}{1-u^2}=x^2.$$
We find that
$$s_2(x)=\frac{1}{x} u^2=\frac{\sqrt{1+2x^2(2r+1)+x^4}-x^2-1}{2r^2x}=\frac{\sqrt{(1+x^2)^2+4rx^2}-x^2-1}{2r^2x}.$$
The inverse of the Pascal-like Sprugnoli array $\left(\frac{1}{1-x}, x, \frac{x(1+rx^2)}{1-x^2}\right)$ is given by
\begin{align*}
w(x)&=1-x\\
s_1(x)&=\frac{\sqrt{(1+x^2)^2+4rx^2}-x^2-rx-1}{2r(x-1)}\\
s_2(x)&=\frac{\sqrt{(1+x^2)^2+4rx^2}-x^2-1}{2r^2x}.\end{align*}
For the Pascal-like Sprugnoli matrix  $\left(\frac{1}{1-x}, \frac{x}{1-x}, \frac{x(1+rx^2)}{1-x^2}\right)$, we find that the inverse is given by
\begin{align*}
w(x)&=\frac{(1-x)(1+2r+x^2-\sqrt{(1+x^2)^2+4rx^2})}{2r}\\
s_1(x)&=\frac{x}{1-x}\\
s_2(x)&=\frac{\sqrt{(1+x^2)^2+4rx^2}-x^2-1}{2r^2x}.\end{align*}
In the particular cases of the Sprugnoli Delannoy matrices, we have the following results.
\begin{scriptsize}
\begin{align*}
\left(\frac{1}{1-x}, x, \frac{x(1+x^2)}{1-x^2}\right)^{-1}&=\left(1-x,\frac{1+x+x^2-\sqrt{1+6x^2+x^4}}{2(1-x)}, \frac{\sqrt{1+6x^2+x^4}-x^2-1}{2x}\right)\\
\left(\frac{1}{1-x}, \frac{x}{1+x}, \frac{x(1+x^2)}{1-x^2}\right)^{-1}&=\left(\frac{(1-x)(3+x^2-\sqrt{1+6x^2+x^4})}{2}, \frac{x}{1-x}, \frac{\sqrt{1+6x^2+x^4}-x^2-1}{2x}\right).\end{align*}
\end{scriptsize}
An aspect of these matrices is that they can bring into consideration sequences not necessarily studied before. For instance, the generating function $\frac{(1-x)(3+x^2-\sqrt{1+6x^2+x^4})}{2}$ expands to give the sequence that begins
$$1, -1, -1, 1, 2, -2, -6, 6, 22, -22, -90, 90, 394, -394, -1806, 1806, \ldots.$$ This is a signed variant of the doubled large Schroeder numbers (\seqnum{A006318}). Note also that the central terms of the Sprugnoli Delannoy array $\left(\frac{1}{1-x}, \frac{x}{1+x}, \frac{x(1+x^2)}{1-x^2}\right)$ is the sequence that begins
$$1, 0, 3, 0, 13, 0, 63, 0, 321, 0, 1683,\ldots,$$ which is an aerated version of the central Delannoy numbers \seqnum{A001850}. The Sprugnoli Delannoy array $\left(\frac{1}{1-x}, x, \frac{x(1+x^2)}{1-x^2}\right)$ has the doubled central Delannoy numbers
$$1, 1, 3, 3, 13, 13, 63, 63, 321, 321, 1683,\ldots$$ as its central sequence.
\begin{example}
We consider the product of the Riordan array $\left(\frac{1+x}{1-x},x\right)^{-1}=\left(\frac{1-x}{1+x},x\right)$ and the Sprugnoli Delannoy array $\left(\frac{1}{1-x}, x, \frac{x(1+x^2)}{1-x^2}\right)$. We recall that this latter matrix begins
$$\left(\begin{array}{cccccccccc}
1 & 0 & 0 & 0 & 0 & 0 & 0 & 0 & 0 \\
1 & 1 & 0 & 0 & 0 & 0 & 0 & 0 & 0 \\
1 & 1 & 1 & 0 & 0 & 0 & 0 & 0 & 0 \\
1 & 1 & 1 & 1 & 0 & 0 & 0 & 0 & 0 \\
1 & 1 & 3 & 1 & 1 & 0 & 0 & 0 & 0 \\
1 & 1 & 3 & 3 & 1 & 1 & 0 & 0 & 0 \\
1 & 1 & 5 & 3 & 5 & 1 & 1 & 0 & 0 \\
1 & 1 & 5 & 5 & 5 & 5 & 1 & 1 & 0 \\
1 & 1 & 7 & 5 & 13 & 5 & 7 & 1 & 1
\end{array}\right).$$ The resulting product begins
$$\left(\begin{array}{cccccccccc}
1 & 0 & 0 & 0 & 0 & 0 & 0 & 0 & 0 \\
-1 & 1 & 0 & 0 & 0 & 0 & 0 & 0 & 0 \\
1 & -1 & 1 & 0 & 0 & 0 & 0 & 0 & 0 \\
-1 & 1 & -1 & 1 & 0 & 0 & 0 & 0 & 0 \\
1 & -1 & 3 & -1 & 1 & 0 & 0 & 0 & 0 \\
-1 & 1 & -3 & 3 & -1 & 1 & 0 & 0 & 0 \\
1 & -1 & 5 & -3 & 5 & -1 & 1 & 0 & 0 \\
-1 & 1 & -5 & 5 & -5 & 5 & -1 & 1 & 0 \\
1 & -1 & 7 & -5 & 13 & -5 & 7 & -1 & 1
\end{array}\right).$$
This is the Sprugnoli array $\left(\frac{1}{1+x}, x, \frac{x(1+x^2)}{1-x^2}\right)$. Its inverse begins
$$\left(\begin{array}{cccccccccc}
1 & 0 & 0 & 0 & 0 & 0 & 0 & 0 & 0 \\
1 & 1 & 0 & 0 & 0 & 0 & 0 & 0 & 0 \\
0 & 1 & 1 & 0 & 0 & 0 & 0 & 0 & 0 \\
0 & 0 & 1 & 1 & 0 & 0 & 0 & 0 & 0 \\
0 & -2 & -2 & 1 & 1 & 0 & 0 & 0 & 0 \\
0 & 0 & -2 & -2 & 1 & 1 & 0 & 0 & 0 \\
0 & 6 & 6 & -4 & -4 & 1 & 1 & 0 & 0 \\
0 & 0 & 6 & 6 & -4 & -4 & 1 & 1 & 0 \\
0 & -22 & -22 & 16 & 16 & -6 & -6 & 1 & 1
\end{array}\right),$$ where once again we see the presence of the large Schroeder numbers.
\end{example}

\begin{example} The inverse of the Pascal-like Sprugnoli array  
$$\left(\frac{1}{1-x}, \frac{x(1+rx)}{1+x}, \frac{x(1+r^2x^2)}{1-x^2}\right)$$ is given by 
\begin{align*}
w(x)&=\frac{(1-x)((r-1)\sqrt{x^4+2x^2(2r^2+1)+1}+x^2(1-r)+2r^2+r+1)}{2r(1+r-(r-1)x^2)}\\
s_1(x)&=x\left(\frac{1}{1-x}-\frac{2r^3(r^2+2)x(1+x)}{(r^2+1)^2\sqrt{-8r^2x^4+4r^2x^2+1}-2r^2x^2+r^4+2r^2-1}\right)\\
s_2(x)&=\frac{\sqrt{x^4+2x^2(2r^2+1)+1}-x^2-1}{2r^2x}.\end{align*}

For instance, the Pascal-like array $$\left(\frac{1}{1-x}, \frac{x(1-x)}{1+x}, \frac{x(1+x^2)}{1-x^2}\right),$$ which begins

\[
\begin{pmatrix}
1 & 0 & 0 & 0 & 0 & 0 & 0 & 0 \\
1 & 1 & 0 & 0 & 0 & 0 & 0 & 0 \\
1 & -1 & 1 & 0 & 0 & 0 & 0 & 0 \\
1 & 1 & 1 & 1 & 0 & 0 & 0 & 0 \\
1 & -1 & 3 & -1 & 1 & 0 & 0 & 0 \\
1 & 1 & 3 & 3 & 1 & 1 & 0 & 0 \\
1 & -1 & 5 & -3 & 5 & -1 & 1 & 0 \\
1 & 1 & 5 & 5 & 5 & 5 & 1 & 1
\end{pmatrix},
\]
has an inverse that begins

\[
\begin{pmatrix}
1 & 0 & 0 & 0 & 0 & 0 & 0 & 0 \\
-1 & 1 & 0 & 0 & 0 & 0 & 0 & 0 \\
-2 & 1 & 1 & 0 & 0 & 0 & 0 & 0 \\
2 & -2 & -1 & 1 & 0 & 0 & 0 & 0 \\
6 & -4 & -4 & 1 & 1 & 0 & 0 & 0 \\
-6 & 6 & 4 & -4 & -1 & 1 & 0 & 0 \\
-22 & 16 & 16 & -6 & -6 & 1 & 1 & 0 \\
22 & -22 & -16 & 16 & 6 & -6 & -1 & 1
\end{pmatrix}.
\]
We have, in this case,
\begin{align*}
w(x)&=\frac{(1-x)(\sqrt{1+6x^2+x^4}-x^2-1)}{2x^2}\\
s_1(x)&=\frac{x(2 \sqrt{-8x^4+4x^2+1}-3x^3-x^2+3x+1}{(1-x)(2\sqrt{-8x^4+4x^2+1}-x^2+1}\\
s_2(x)&=\frac{\sqrt{1+6x^2+x^4}-x^2-1}{2x}.\end{align*}
\end{example}

The inverse of the general Pascal-like Sprugnoli array 
$$\left(\frac{1}{1-x}, \frac{x(1+sx)}{1+x}, \frac{x(1+rx^2)}{1-x^2}\right)$$ is given by 
\begin{align*}
w(x)&=\frac{2(r+1)(x-1)}{(s-1) \sqrt{-8rx^4+4rx^2+1}+x^2(s-1)-2r-s-1}\\
s_1(x)&=\frac{x((r+1)^2 \sqrt{-8rx^4+4rx^2+1}+2rsx^3(r+2)-2rx^2-2rsx(r+2)+r^2+2r-1)}{(1-x)((r+1)^2 \sqrt{-8rx^4+4rx^2+1}-2rx^2+r^2+2r-1)}\\
s_2(x)&=\frac{(1-x)((r-1)\sqrt{x^4+2x^2(2r^2+1)+1}+x^2(1-r)+2r^2+r+1)}{2r(1+r-(r-1)x^2)}.\end{align*}
\begin{example}
The Pascal-like Sprugnoli array $\left(\frac{1}{1-x}, \frac{x(1+2x)}{1+x}, \frac{x(1+x^2)}{1-x^2}\right)$, which begins 
\[
\begin{pmatrix}
1 & 0 & 0 & 0 & 0 & 0 & 0 & 0 \\
1 & 1 & 0 & 0 & 0 & 0 & 0 & 0 \\
1 & 2 & 1 & 0 & 0 & 0 & 0 & 0 \\
1 & 1 & 1 & 1 & 0 & 0 & 0 & 0 \\
1 & 2 & 3 & 2 & 1 & 0 & 0 & 0 \\
1 & 1 & 3 & 3 & 1 & 1 & 0 & 0 \\
1 & 2 & 5 & 6 & 5 & 2 & 1 & 0 \\
1 & 1 & 5 & 5 & 5 & 5 & 1 & 1
\end{pmatrix}
\] 
has an inverse that begins 
\[
\begin{pmatrix}
1 & 0 & 0 & 0 & 0 & 0 & 0 & 0 \\
-1 & 1 & 0 & 0 & 0 & 0 & 0 & 0 \\
1 & -2 & 1 & 0 & 0 & 0 & 0 & 0 \\
-1 & 1 & -1 & 1 & 0 & 0 & 0 & 0 \\
0 & 2 & -1 & -2 & 1 & 0 & 0 & 0 \\
0 & 0 & 1 & -1 & -1 & 1 & 0 & 0 \\
2 & -8 & 4 & 6 & -3 & -2 & 1 & 0 \\
-2 & 2 & -4 & 4 & 3 & -3 & -1 & 1
\end{pmatrix}.
\]
The first column of this matrix, which begins
$$1, -1, 1, -1, 0, 0, 2, -2, -6, 6, 26, -26, -114, 114, 526, -526, -2502, 2502,\ldots$$ 
has generating function given by 
$$\frac{4(1-x)}{5-x^2-\sqrt{1+6x^2+x^4}}.$$ 
The bisection 
$$1, 1, 0, 2, -6, 26, -114, 526, -2502, 12194, -60570,\ldots$$ has a generating function given by 
$$\frac{4}{5-x-\sqrt{1+6x+x^2}},$$ or equivalently,
$$\cfrac{1}{1-x+\cfrac{x^2}{1+3x-\cfrac{2x^2}{1+3x-\cfrac{2x^2}{1+3x-\cdots}}}}.$$
The central sequence of $\left(\frac{1}{1-x}, \frac{x(1+2x)}{1+x}, \frac{x(1+x^2)}{1-x^2}\right)$ begins 
$$1, 2, 3, 6, 13, 26, 63, 126, 321, 642, 1683, 3366,\ldots.$$ The even bisection of this sequence gives the central Delannoy numbers \seqnum{A001850} and hence the generating function of this central sequence is given by
$$\frac{1+2x}{\sqrt{1-6x^2+x^4}}.$$
\end{example}

\section{Further examples}
\begin{example}
When $r=-1$, we obtain $\frac{x(1+rx^2)}{1-x^2}=\frac{x(1-x^2)}{1-x^2}=x$. We are led to consider the Pascal-like Sprugnoli arrays
$\left(\frac{1}{1-x}, x, x\right)$, which begins
$$\left(\begin{array}{cccccccccc}
1 & 0 & 0 & 0 & 0 & 0 & 0 & 0 & 0 \\
1 & 1 & 0 & 0 & 0 & 0 & 0 & 0 & 0 \\
1 & 1 & 1 & 0 & 0 & 0 & 0 & 0 & 0 \\
1 & 1 & 1 & 1 & 0 & 0 & 0 & 0 & 0 \\
1 & 1 & 1 & 1 & 1 & 0 & 0 & 0 & 0 \\
1 & 1 & 1 & 1 & 1 & 1 & 0 & 0 & 0 \\
1 & 1 & 1 & 1 & 1 & 1 & 1 & 0 & 0 \\
1 & 1 & 1 & 1 & 1 & 1 & 1 & 1 & 0 \\
1 & 1 & 1 & 1 & 1 & 1 & 1 & 1 & 1
\end{array}\right),$$ and
$\left(\frac{1}{1-x}, \frac{x}{1-x}, x\right)$, that begins

$$\left(\begin{array}{cccccccccc}
1 & 0 & 0 & 0 & 0 & 0 & 0 & 0 & 0 \\
1 & 1 & 0 & 0 & 0 & 0 & 0 & 0 & 0 \\
1 & 0 & 1 & 0 & 0 & 0 & 0 & 0 & 0 \\
1 & 1 & 1 & 1 & 0 & 0 & 0 & 0 & 0 \\
1 & 0 & 1 & 0 & 1 & 0 & 0 & 0 & 0 \\
1 & 1 & 1 & 1 & 1 & 1 & 0 & 0 & 0 \\
1 & 0 & 1 & 0 & 1 & 0 & 1 & 0 & 0 \\
1 & 1 & 1 & 1 & 1 & 1 & 1 & 1 & 0 \\
1 & 0 & 1 & 0 & 1 & 0 & 1 & 0 & 1
\end{array}\right).$$ The row sums of this last matrix begin
$$1, 2, 2, 4, 3, 6, 4, 8, 5, 10, 6,\ldots$$ or \seqnum{A029578}, with generating function $\frac{1+2x}{(1-x^2)^2}$. The diagonal sums, which begin $$1, 1, 2, 1, 3, 2, 4, 2, 5, 3, 6,\ldots$$ interleave the positive integers with the doubled positive integers.
\end{example}
\begin{example} When $r=-2$, we obtain the Sprugnoli arrays $$\left(\frac{1}{1-x}, x, \frac{x(1-2x^2)}{1-x^2}\right)\quad\text{and}\quad
 \left(\frac{1}{1-x}, \frac{x}{1+x}, \frac{x(1-2x^2)}{1-x^2}\right).$$
 They begin, respectively,
 $$\left(\begin{array}{cccccccccc}
 1 & 0 & 0 & 0 & 0 & 0 & 0 & 0 & 0 \\
1 & 1 & 0 & 0 & 0 & 0 & 0 & 0 & 0 \\
1 & 1 & 1 & 0 & 0 & 0 & 0 & 0 & 0 \\
1 & 1 & 1 & 1 & 0 & 0 & 0 & 0 & 0 \\
1 & 1 & 0 & 1 & 1 & 0 & 0 & 0 & 0 \\
1 & 1 & 0 & 0 & 1 & 1 & 0 & 0 & 0 \\
1 & 1 & -1 & 0 & -1 & 1 & 1 & 0 & 0 \\
1 & 1 & -1 & -1 & -1 & -1 & 1 & 1 & 0 \\
1 & 1 & -2 & -1 & -2 & -1 & -2 & 1 & 1
\end{array}\right)$$ and
 $$\left(\begin{array}{cccccccccc}
 1 & 0 & 0 & 0 & 0 & 0 & 0 & 0 & 0 \\
1 & 1 & 0 & 0 & 0 & 0 & 0 & 0 & 0 \\
1 & 0 & 1 & 0 & 0 & 0 & 0 & 0 & 0 \\
1 & 1 & 1 & 1 & 0 & 0 & 0 & 0 & 0 \\
1 & 0 & 0 & 0 & 1 & 0 & 0 & 0 & 0 \\
1 & 1 & 0 & 0 & 1 & 1 & 0 & 0 & 0 \\
1 & 0 & -1 & 0 & -1 & 0 & 1 & 0 & 0 \\
1 & 1 & -1 & -1 & -1 & -1 & 1 & 1 & 0 \\
1 & 0 & -2 & 0 & -2 & 0 & -2 & 0 & 1
\end{array}\right).$$ The row sums of this last matrix begin
$$1, 2, 2, 4, 2, 4, 0, 0, -4, -8, -8, -16, -8, -16, 0, 0, 16, 32, 32, 64, 32,\ldots,$$ with generating function
$\frac{1+2x}{1-2x^2+2x^4}$.
\end{example}

\section{Row sums of matrices modulo $2$}
We conjecture the following results.
When $r$ is even, the row sums of the modulo $2$ of the Pascal-like Sprugnoli arrays $\left(\frac{1}{1-x},\frac{x}{1+x}, \frac{x(+rx^2)}{1-x^2}\right)$ are given by the Gould's sequence \seqnum{A001316} which begins
$$1, 2, 2, 4, 2, 4, 4, 8, 2, 4, 4,\ldots.$$
When $r$ is odd, the row sums of the modulo $2$ of the Pascal-like Sprugnoli arrays $\left(\frac{1}{1-x},\frac{x}{1+x}, \frac{x(+rx^2)}{1-x^2}\right)$ begin
$$1, 2, 2, 4, 3, 6, 4, 8, 5, 10, 6,\ldots.$$ This is essentially \seqnum{A029578}, the positive integers interlaced with the positive even integers.

When $r$ is even,  the row sums of the modulo $2$ of the Pascal-like Sprugnoli arrays $\left(\frac{1}{1-x},x, \frac{x(+rx^2)}{1-x^2}\right)$ are given by the generalized Gould's sequence \seqnum{A114212} which begins
$$1, 2, 3, 4, 4, 4, 6, 8, 6, 4, 6,\ldots.$$
When $r$ is odd, the row sums of the modulo $2$ of the Pascal-like Sprugnoli arrays $\left(\frac{1}{1-x},x, \frac{x(+rx^2)}{1-x^2}\right)$ are given by the natural numbers \seqnum{A000027}
$$1,2,3,4,5,6,7,8,9,10,\ldots.$$

We further conjecture that for all values of $r$, the row sums of the modulo $2$ of the inverse array $\left(\frac{1}{1-x}, \frac{x}{1+x}, \frac{x(1-rx^2)}{1-x^2}\right)^{-1}$ are given by Gould's sequence.

\section{Conclusion} We have defined what we mean by a Pascal-like number triangle, and we have given two paramaterized families of Sprugnoli arrays which are Pascal-like. We have examined their row sums, diagonal sums, and in the case of the Sprugnoli Delannoy matrices, we have given their inverses. We have touched briefly of some central sequences. We have conjectured the form of the row sums of these Pascal-like matrices when they are taken modulo two. Many well-known sequences have emerged, some as bisections of row sums and diagonal sums. It is hoped that further studies of these Sprugnoli arrays will yield further interesting results.

\bigskip
\hrule
\bigskip
\noindent 2020 {\it Mathematics Subject Classification}:
Primary 15B36; Secondary 05A15, 11B83, 15A30, 20H25.

\noindent \emph{Keywords:} Pascal triangle, Pascal-like matrix, Riordan group, Sprugnoli group, generating function.

\bigskip
\hrule
\bigskip
\noindent (Concerned with sequence
\seqnum{A000027},
\seqnum{A000045},
\seqnum{A000073},
\seqnum{A000129},
\seqnum{A001316},
\seqnum{A001850},
\seqnum{A006318},
\seqnum{A008288},
\seqnum{A029578},
\seqnum{A047081},
\seqnum{A051159},
\seqnum{A052936},
\seqnum{A053602},
\seqnum{A060546},
\seqnum{A078057},
\seqnum{A114212},
\seqnum{A122367}, and
\seqnum{A153367}
).

\end{document}